\documentclass[reqno]{amsart}
\usepackage{amscd}
\usepackage{epsfig}  
\usepackage[utf8]{inputenc}
\usepackage{latexsym,amssymb}
\usepackage{amsfonts}
\usepackage[T1]{fontenc}
\usepackage{amsthm}
\usepackage{amsmath,multicol,multirow,eso-pic}
\usepackage{graphicx}
\usepackage{color}
\usepackage{arcs}
\usepackage{float}
\usepackage{tikz}
\usepackage{tkz-tab}
\usepackage{array}

\theoremstyle{plain}
\newtheorem{thm}{Theorem}[section]

\newtheorem{lem}[thm]{Lemma}
\newtheorem{prop}[thm]{Proposition}
\theoremstyle{definition}
\newtheorem{defn}{Definition}[section]

\numberwithin{equation}{section}

\begin{document}

\title[]{REPRESENTATIONS OF THE NECKLACE BRAID GROUP ${\mathcal{NB}}_n$ OF DIMENSION 4 ($n=2,3,4$)}

\author{Taher I. Mayassi \and Mohammad N. Abdulrahim }

\address{Taher I. Mayassi\\
         Department of Mathematics and Computer Science\\
         Beirut Arab University\\
        P.O. Box 11-5020, Beirut, Lebanon}
\email{tim187@student.bau.edu.lb}

\address{Mohammad N. Abdulrahim\\
         Department of Mathematics and Computer Science\\
         Beirut Arab University\\
         P.O. Box 11-5020, Beirut, Lebanon}
\email{mna@bau.edu.lb}

\begin{abstract}
We consider the irreducible representations each of dimension 2 of the necklace braid group $\mathcal{NB}_n$ ($n=2,3,4$). We then consider the tensor product of the representations of $\mathcal{NB}_n$ ($n=2,3,4$) and determine necessary and sufficient condition under which the constructed representations are irreducible. Finally, we determine conditions under which the irreducible representations of $\mathcal{NB}_n$ ($n=2,3,4$) of degree 2 are unitary relative to a hermitian positive definite matrix.
\end{abstract}
\maketitle

\section{Introduction}
\bigskip
In \cite{Bulli}, Alex Bullivant, Andrew Kimball, Paul Martin and Eric C. Rowell studied representations of the necklace braid group $\mathcal{NB}_n$, especially those obtained as extensions of representations of the braid group $B_n$ and the loop braid group $LB_n$ (see \cite{Brui} and\cite{Kad}). They showed that any irreducible $B_n $ representation extends to $\mathcal{NB}_n$  in a standard way. Moreover, they proved that any local representation of $B_n$, coming from a braided vector space, can be extended to $\mathcal{NB}_n$. \\
A link $\mathcal{L}_n=K_0\cup K_1\cup\cdots\cup K_n$ is called a necklace if:
\begin{itemize}
\item $K_0$ is an Euclidean circle of center $O$ and of radius 1
\item each $K_i$ is an Euclidean circle whose center $O_i$ belongs to $K_0$ and radius $r_i$ such that $0<r_i<\frac{1}{2}$ for $0<i\leqslant n$
\item the plane of each $K_i$ is the one containing the line $(OO_i)$ and perpendicular to the plane of $K_0$.
\item If $O_i=O_j$ then $r_i\neq r_j$
\begin{figure}[h]
\hfil\epsfig{file=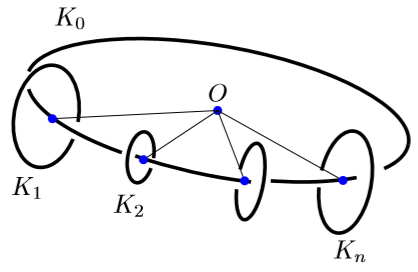, width=2in}\hfil
\caption{\label{fig} Necklace $\mathcal{L}_n$}
\end{figure}
\end{itemize}
\renewcommand{\thefootnote}{}
\footnote{\textit{Key words and phrases.} Necklace braid group, tensor, irreducible, unitary.}
\footnote{\textit{Mathematics Subject Classification.} Primary: 20F36.}
\vskip 0.1in 
\indent The motion group $\mathcal{NB}_n$, {\em the necklace braid group}, as described in \cite{Belli} is identified with the fundamental group of the configuration space $\mathcal{L}_n$. The group $\mathcal{NB}_n$ is generated by the elements $\sigma_1,\cdots,\sigma_n$ and $\tau$ where $\sigma_i$ is the motion, up to homotopy, of passing the $i$ th circle through the $(i+1)$th circle, while $\tau$ corresponds to shifting each circle one position in the counterclockwise direction.\\
In section 2, we consider the irreducible representations defined on the necklace braid group ${\mathcal{NB}}_n$ ($n=2,3,4$) each of dimension 2. In section 3, we construct the tensor product of  representations on ${\mathcal{NB}}_n$. Then we discuss the irreducibility of the tensor product of the representations of $\mathcal{NB}_n$ ($n=2,3,4$). Theorem 3.4 gives necessary and sufficient condition for the irreducibility of  tensor product of the representations of $\mathcal{NB}_4$. Theorem 3.16 provides necessary and sufficient conditions under which the tensor product of the representations of $\mathcal{NB}_3$ is irreducible. Theorem 3.17 gives necessary and sufficient condition for the irreducibility of the tensor product of the representations of $\mathcal{NB}_2$.  In section 4, we prove that the irreducible representations of dimension 2 of $\mathcal{NB}_n$ are unitary relative to herimitian positive definite matrix in the case $n=3,4$ (see Proposition 4.2, Proposition 4.3 and Proposition 4.4).\\
\begin{defn}\cite{Belli}
The necklace braid group $\mathcal{NB}_n$ is identified with the fundamental group of the configuration space of $\mathcal{L}_n$.
\end{defn}
The following theorem gives a presentation of the necklace braid group $\mathcal{NB}_n$ by generators and relations. 
\begin{thm}\cite{Belli}
The necklace braid group ${\mathcal{NB}}_n$ has a presentation by generators $\sigma_1,...,\sigma_n,\tau$ and relations as follows:
\begin{itemize} 
\item[(B1)] $\sigma_i \sigma_{i+1} \sigma_i  = \sigma_{i+1} \sigma_i \sigma_{i+1}$ for $1\leqslant i\leqslant n$;
\item[(B2)] $\sigma_i \sigma_j = \sigma_j \sigma_i   \text{ for } |i-j|\neq 1$;
\item[(N1)] $\tau\sigma_i\tau^{-1}=\sigma_{i+1}$  for $1\leqslant i \leqslant n~( \text{ mod } n)$;
\item[(N2)] $\tau^{2n}=1$
\end{itemize}
\end{thm}
\vspace{-0.09cm}
Here indices are taken modulo $n$, with $\sigma_{n+1}=\sigma_1$ and $\sigma_0=\sigma_n$. The relations (B1) and (B2) are those for the braid group $B_n$ (see \cite{Bir}).
\vspace{.4cm}
\section{Irreducible Representations of dimension 2 of $\mathcal{NB}_n$ ($n=2,3,4$)}
\bigskip
Consider the necklace braid group $\mathcal{NB}_n$ for $n=2,3,4$. We have the following proposition (see \cite{Bulli}).
\begin{prop}
Any irreducible representation of dimension 2 of  $\mathcal{NB}_n$ ($n=2,3,4$) is isomorphic to the representation $\rho=\rho(T,t,a,c,d)$ that is defined by:
$$\rho(\tau)=\left(\begin{matrix}
Tt&0\\
0&t
\end{matrix}\right) \text{ and } 
\rho(\sigma_i)=\left(\begin{matrix}
a&T^{i-1}\\
cT^{1-i}&d
\end{matrix}\right) $$

with the following conditions:

\vspace{.3cm}
\renewcommand{\arraystretch}{2}
\begin{tabular}{|c|c|c|}
\hline 
$n$ & $T$ & Conditions \\ 
\hline 
2 & $T=-1$ &  $c=a^2-ad+d^2$, $c\neq0$ and $a\neq d$ \\
\hline
3 & 
$T=e^{\pm i2\pi/3}$ & $a=\omega d$, $c\neq0$, $c\neq \omega d^2$ and $\omega=e^{\pm i\pi/3}$\\  
\hline
 \multirow{2}{*}{4}& $T=-1$ &  $c=a^2-ad+d^2$, $c\neq0$ and $a\neq d$ \\
\cline{2-3} 
 & $T=\pm i$ &  $c=-d^2$, $a=d\neq0$ \\
\hline 
\end{tabular}

\smallskip
\noindent Here $t$ is any $(2n)^{th}$ root of unity.
  
\end{prop}

\begin{proof}
Let $(\rho,V)$ be an irreducible representation of dimension 2 of $\mathcal{NB}_n$ (for $n=2,3,4$).
From the fact that $\tau$ has order $2n$, we may choose a basis for $V$ such that $\rho(\tau) =
\left(\begin{matrix}
t_1&0\\
0&t_2
\end{matrix}\right)$,
where $t_1$ and $t_2$ are $(2n)$-th roots of unity. Since $\rho$ is irreducible representation, it follows that $t_1\neq t_2$. Similarly we have that $\rho(\sigma_1)$ is neither upper nor lower triangular, because $(1,0)$ or $(0,1)$ would generate an invariant subspace. Due to rescaling, we may assume that $\rho(\sigma_1) =
\left(\begin{matrix}
a&1\\
c&d
\end{matrix}\right)$. Since
$\rho(\sigma_1)$ is neither diagonal nor triangular, we have that $c\neq 0$. To require invertiblity, we assume that $ad\neq c$.
\newline Using $\rho(\sigma_1)=\left(\begin{matrix}
a&1\\
c&d
\end{matrix}\right)$ and the condition (N1): $\tau\sigma_i\tau^{-1}=\sigma_{i+1}$, we get
$$\rho(\sigma_i)=\left(\begin{matrix}
a&T^{i-1}\\
cT^{1-i}&d
\end{matrix}\right),$$
where $T=t_1t_2^{-1}$. Note that $T\neq1$ because $t_1\neq t_2$.
Set $t_2=t$ then $t_1=Tt$ and $\rho(\tau)$ has the form $\left(\begin{matrix}
Tt&0\\
0&t
\end{matrix}\right)$.
\newline Since $\sigma_{n+1}=\sigma_1$, we have $T^{n}=T^{-n}=1$. Thus $T$ is a primitive $n$-th root of unity.
\newline Now, we check the conditions (B1): $\sigma_i\sigma_{i+1}\sigma_i=\sigma_{i+1}\sigma_i\sigma_{i+1}$. By Lemma 1.2 in \cite{Bulli}, it is sufficient to check (B1) for $i=1$ (i.e. $\sigma_1\sigma_2\sigma_1=\sigma_2\sigma_1\sigma_2$ ).
\newline
By direct calculations, we get
$$\rho(\sigma_1)\rho(\sigma_2)\rho(\sigma_1)=\rho(\sigma_2)\rho(\sigma_1)\rho(\sigma_2)\Leftrightarrow(a^2-ad+d^2)T+c(1+T+T^2)=0.$$
For $n=2$, we have that $T=-1$ and $c=a^2-ad+d^2$. Under the assumption that $ad\neq c$, we have $a\neq d$.
\newline 
For $n=3$, we have that $T=e^{\pm2\pi i/3}$. Then $a=\omega d$ where $\omega=e^{\pm\pi i/3}$. Under the assumption that $ad\neq c$, we have $c\neq\omega d^2$.
\newline
For $n=4$, we have that $T=-1, \pm i$. 
\begin{itemize}
\item If $T=-1$ then $c=a^2-ad+d^2$ with $a\neq d$. 
\item If $T=\pm i$, then $c=-a^2+ad-d^2$. 
\end{itemize}
It remains to verify the condition (B2): $\sigma_i\sigma_j=\sigma_j\sigma_i$ for $|i-j|>1$ when $n=4$. By Lemma 1.2 in \cite{Bulli} we just check (B2) for $i=1 $ and $j=3$.\\
If $T=-1$ then it is clear that $\rho(\sigma_1)\rho(\sigma_3)=\rho(\sigma_3)\rho(\sigma_1)$.\\
If $T=\pm i$ then
\begin{align*}
\rho(\sigma_1)\rho(\sigma_3)=\rho(\sigma_3)\rho(\sigma_1)&\Longleftrightarrow \begin{pmatrix}
a&1\\
c&d
\end{pmatrix}
\begin{pmatrix}
a&-1\\
-c&d
\end{pmatrix}
=
\begin{pmatrix}
a&-1\\
-c&d
\end{pmatrix}
\begin{pmatrix}
a&1\\
c&d
\end{pmatrix}\\
&\Longleftrightarrow \begin{pmatrix}
a^2-c&-a+d\\
ca-dc&-c+d^2
\end{pmatrix}
=
\begin{pmatrix}
a^2-c&a-d\\
-ca+dc&-c+d^2
\end{pmatrix}\\
&\Longleftrightarrow a=d
\end{align*}
It follows that $c=a^2+ad-d^2=-d^2$.
\end{proof}
\vspace{.4cm}
\section{Representations of dimension 4 of $\mathcal{NB}_n$ ($n=2,3,4$)}
\bigskip
Consider two irreducible representations $\rho_1=\rho(T_1,t_1,a_1,c_1,d_1)$ and $\rho_2=\rho(T_2,t_2,a_2,c_2,d_2)$ of $\mathcal{NB}_n$ ($n=2,3,4$) each of dimension 2.
$$\rho_1(\tau)=\left(\begin{matrix}
T_1t_1&0\\
0&t_1
\end{matrix}\right) \text{ and } 
\rho_1(\sigma_i)=\left(\begin{matrix}
a_1&T_1^{i-1}\\
c_1T_1^{1-i}&d_1
\end{matrix}\right) $$
$$\rho_2(\tau)=\left(\begin{matrix}
T_2t_2&0\\
0&t_2
\end{matrix}\right) \text{ and } 
\rho_2(\sigma_i)=\left(\begin{matrix}
a_2&T_2^{i-1}\\
c_2T_1^{1-i}&d_2
\end{matrix}\right) $$
\begin{defn} Consider the tensor product $\rho_1\otimes\rho_2$ given by $\rho_1\otimes\rho_2(\alpha)=\rho_1(\alpha)\otimes\rho_2(\alpha),$ where $\alpha$ is a generator of $\mathcal{NB}_n$ ($n=2,3,4$). We get the following matrices for the generators $\tau$ and $\sigma_1$.
$$\rho(\tau)=t_1t_2\left(\begin{matrix}
T_1T_2&0&0&0\\
0&T_1&0&0\\
0&0&T_2&0\\
0&0&0&1
\end{matrix}\right)~~\text{ and  }~~
\rho(\sigma_1)=\left(\begin{matrix}
a_1a_2&a_1&a_2&1\\
a_1c_2&a_1d_2&c_2&d_2\\
c_1a_2&c_1&d_1a_2&d_1\\
c_1c_2&c_1d_2&d_1c_2&d_1d_2
\end{matrix}\right).$$
\end{defn}
\medskip
\subsection{Representations of $\mathcal{NB}_4$}

In this section, we study the irreduciblty of the representation $\rho$ of the necklace braid group $\mathcal{NB}_4$. Actually, we have four cases: (3.1.1) $T_1=T_2=-1$, (3.1.2) $T_1=T_2=\pm i$, (3.1.3) $T_1=-T_2=\pm i$ and (3.1.4) $T_1=-1, T_2=\pm i$.\\\\
In what follows, suppose that $a_1a_2d_1d_2\neq0$. 

\subsubsection{Case $T_1=T_2=-1$}

Direct computations show that the representation $\rho$ is given by
$$\rho(\tau)=
t_1t_2\left(\begin{matrix}
1&0&0&0\\
0&-1&0&0\\
0&0&-1&0\\
0&0&0&1
\end{matrix}\right),
~
\rho(\sigma_1)=\rho(\sigma_3)=\left(\begin{matrix}
a_1a_2&a_1&a_2&1\\
a_1c_2&a_1d_2&c_2&d_2\\
c_1a_2&c_1&d_1a_2&d_1\\
c_1c_2&c_1d_2&d_1c_2&d_1d_2
\end{matrix}\right)$$
and
$$
\rho(\sigma_2)=\rho(\sigma_4)=\left(\begin{matrix}
a_1a_2&-a_1&-a_2&1\\
-a_1c_2&a_1d_2&c_2&-d_2\\
-c_1a_2&c_1&d_1a_2&-d_1\\
c_1c_2&-c_1d_2&-d_1c_2&d_1d_2
\end{matrix}\right).$$

The eigenvalues of $\rho(\tau)$ are $\lambda_1=t_1t_2$  and $\lambda_2=-t_1t_2$. Both are of 
multiplicities 2. The corresponding eigenvectors are $\alpha_1e_1+\alpha_4e_4=(\alpha_1,0,0,\alpha_4) $ for $\lambda_1=t_1t_2$ and $\alpha_2e_2+\alpha_3e_3=(0,\alpha_2,\alpha_3,0)$ for $\lambda_2=-t_1t_2$, where $\{\alpha_1,\alpha_2,\alpha_3,\alpha_4\}\subset\mathbb{C}$.\\\\
We now determine conditions under which the representation $\rho$ is irreducible.
\begin{prop}
The representation $\rho:{\mathcal{NB}}_4\to GL(4,\mathbb{C})$ has no non-trivial proper invariant subspaces of dimension 
1 if and only if $$a_1a_2\neq d_1d_2~\text{ and }~a_1d_2\neq a_2d_1.$$
\end{prop}
\begin{proof}
The subspaces of dimension 1 that are invariant under $\rho(\tau)$ are those spanned by one of the 
following vectors: $e_1,~e_2,~e_3,~e_4,~e_1+ye_4~\text{ and }e_2+x e_3~~\text{ for }x\neq0$ and $y\neq0$.\\
$\rho(\sigma_1)(e_1)=\left(\begin{matrix}
a_1a_2\\
a_1c_2\\
c_1a_2\\
c_1c_2
\end{matrix}\right)\not\in\langle e_1\rangle$ because,  $c_1c_2\neq0$.
So, $\langle e_1\rangle $ is not invariant.

$\rho(\sigma_1)(e_2)=\left(\begin{matrix}
a_1\\
a_1d_2\\
c_1\\
c_1d_2
\end{matrix}\right)\not\in\langle e_2\rangle$ because, $c_1\neq0$.
So, $\langle e_2\rangle$ is not invariant.

$\rho(\sigma_1)(e_3)=\left(\begin{matrix}
a_2\\
c_2\\
d_1a_2\\
d_1c_2
\end{matrix}\right)\not\in\langle e_3\rangle$ because,  $c_2\neq0$.
Thus, $\langle e_3\rangle$ is not invariant.

$\rho(\sigma_1)(e_4)=\left(\begin{matrix}
1\\
d_2\\
d_1\\
d_1d_2
\end{matrix}\right)\not\in\langle e_4\rangle$.
So, $\langle e_4\rangle$ is not invariant.\\

Now consider the subspaces of the form $\langle e_1+ye_4\rangle$ with $y\neq0$\\
Since $\rho(\sigma_i)(e_1+ye_4)=\left(\begin{matrix}
a_1a_2+y\\
\pm a_1c_2\pm d_2y\\
\pm c_1a_2\pm d_1y\\
c_1c_2+d_1d_2y
\end{matrix}\right)$, it follows that $\langle e_1+ye_4\rangle$ is invariant under $\rho(\sigma_i)$ for $1\leqslant i\leqslant n.$ It follows that
$\left(\begin{matrix}
a_1a_2+y\\
\pm a_1c_2\pm d_2y\\
\pm c_1a_2\pm d_1y\\
c_1c_2+d_1d_2y
\end{matrix}\right)=
\left(\begin{matrix}
\alpha\\
0\\
0\\
\alpha y
\end{matrix}\right)$ for some $\alpha\in\mathbb{C}\setminus\{0\}$.\\
 This is equivalent to the system:

\begin{align*}
 &\left\lbrace\begin{array}{l}
(a_1a_2+y)y= c_1c_2+d_1d_2y \\
 a_1c_2+d_2y=0\\
   c_1a_2+d_1y=0
\end{array}\right.\\\\
\Longleftrightarrow &
\left\lbrace \begin{array}{l}
y^2+(a_1a_2-d_1d_2)y-c_1c_2=0\\
y=-\dfrac{a_1c_2}{d_2}\\
y=-\dfrac{c_1a_2}{d_1}
\end{array}\right.
\end{align*}
\begin{align*}
\Longleftrightarrow &
\left\lbrace \begin{array}{l}
\dfrac{a_1^2c_2^2}{d_2^2}-\dfrac{a_1^2a_2c_2}{d_2}+d_1a_1c_2-c_1c_2=0\\
-\dfrac{a_1c_2}{d_2}=
-\dfrac{c_1a_2}{d_1}
\end{array}\right.\\\\
\Longleftrightarrow &\left\lbrace \begin{array}{l}
a_1^2c_2-a_1^2a_2d_2+d_1a_1d_2^2-c_1d_2^2=0\\
a_1d_1c_2=a_2d_2c_1
\end{array}\right.\\\\
\Longleftrightarrow &\left\lbrace \begin{array}{l}
a_1^2(c_2-a_2d_2)-d_2^2(c_1-a_1d_1)=0\\
a_1d_1(a_2^2-a_2d_2+d_2^2)=a_2d_2(a_1^2-a_1d_1+d_1^2)
\end{array}\right.\\\\
\Longleftrightarrow & \left\lbrace \begin{array}{l}
a_1^2(a_2-d_2)^2-d_2^2(a_1-d_1)^2=0\\
a_1d_1a_2^2-a_1d_1a_2d_2+a_1d_1d_2^2=a_2d_2a_1^2-a_2d_2a_1d_1+a_2d_2d_1^2
\end{array}\right.\\\\
\Longleftrightarrow & \left\lbrace \begin{array}{l}
a_1^2(a_2-d_2)^2-d_2^2(a_1-d_1)^2=0\\
a_1d_1a_2^2-2a_1d_1a_2d_2+a_1d_1d_2^2=a_2d_2a_1^2-2a_2d_2a_1d_1+a_2d_2d_1^2
\end{array}\right.\\\\
\Longleftrightarrow&\left\lbrace \begin{array}{l}
a_1^2(a_2-d_2)^2-d_2^2(a_1-d_1)^2=0\\
a_1d_1(a_2^2-2a_2d_2+d_2^2)=a_2d_2(a_1^2-2a_1d_1+d_1^2)
\end{array}\right.\\\\
\Longleftrightarrow&\left\lbrace \begin{array}{l}
a_1^2(a_2-d_2)^2=d_2^2(a_1-d_1)^2\\
a_1d_1(a_2-d_2)^2=a_2d_2(a_1-d_1)^2
\end{array}\right.\\\\
\Longleftrightarrow &\left\lbrace \begin{array}{l}
\dfrac{a_1^2}{d_2^2}=\dfrac{(a_1-d_1)^2}{(a_2-d_2)^2}\\
\dfrac{a_1d_1}{a_2d_2}=\dfrac{(a_1-d_1)^2}{(a_2-d_2)^2}
\end{array}\right.\\\\
\Longleftrightarrow &\dfrac{a_1}{d_2}=\dfrac{d_1}{a_2}\\\\
\Longleftrightarrow & a_1a_2=d_1d_2
 \end{align*}
We consider the subspaces of the form $\langle e_2+xe_3\rangle$ with $x\neq0$.\\\\
$\rho(\sigma_j)(e_2+xe_3)=\left(\begin{matrix}
\pm a_1\pm a_2x\\
a_1d_2+c_2x\\
c_1+d_1a_2x\\
\pm c_1d_2\pm d_1c_2x
\end{matrix}\right)$ for $1\leqslant j\leqslant 4$.\\
If we assume that the subspace $\langle e_2+xe_3\rangle$ is invariant, then 
$\left(\begin{matrix}
a_1+a_2x\\
a_1d_2+c_2x\\
c_1+d_1a_2x\\
c_1d_2+d_1c_2x
\end{matrix}\right)=
\left(\begin{matrix}
0\\
\alpha\\
\alpha x\\
0
\end{matrix}\right)$ for some $\alpha\in\mathbb{C}^*$.
This is equivalent to:
\begin{align*}
&\left\lbrace \begin{array}{l}
 a_1+a_2x=0,\\  c_1d_2+d_1c_2x=0, \\ a_1d_2+c_2x=\alpha \text{ and } \\
 c_1+d_1a_2x=\alpha x
 \end{array}
 \right.\\\\
 \Longleftrightarrow & x=-\dfrac{a_1}{a_2},~x=-\dfrac{c_1d_2}{d_1c_2},~\text{ and }~ c_2x^2+(a_1d_2-d_1a_2)x-c_1=0\\\\
 \Longleftrightarrow& c_2\dfrac{a_1^2}{a_2^2}-\dfrac{a_1^2d_2}{a_2}+d_1a_1-c_1=0 ~\text{ and }~ a_1c_2d_1=a_2c_1d_2\\\\
\Longleftrightarrow & c_2a_1^2-a_1^2d_2a_2+d_1a_1a_2^2-c_1a_2^2=0 ~\text{ and }~ a_1d_1(a_2^2-a_2d_2+d_2^2)=a_2d_2(a_1^2-a_1d_1+d_1^2)\\\\
\Longleftrightarrow & a_1^2(c_2-d_2a_2)=a_2^2(-d_1a_1+c_1) ~\text{ and }~ a_1d_1(a_2^2-2a_2d_2+d_2^2)=a_2d_2(a_1^2-2a_1d_1+d_1^2)\\\\
\Longleftrightarrow& a_1^2(a_2-d_2)^2=a_2^2(a_1-d_1)^2 ~\text{ and }~ a_1d_1(a_2-d_2)^2=a_2d_2(a_1-d_1)^2\\\\
\Longleftrightarrow  &\dfrac{a_1^2}{a_2^2}=\dfrac{(a_1-d_1)^2}{(a_2-d_2)^2}=\dfrac{a_1d_1}{a_2d_2}
\\\\
\Longleftrightarrow& \dfrac{a_1}{a_2}=\dfrac{d_1}{d_2}\\\\
\Longleftrightarrow & a_1d_2=a_2d_1.
\end{align*}
\end{proof}

\begin{prop} The representation $\rho:{\mathcal{NB}}_4\to GL(4,\mathbb{C})$ has no non-trivial proper invariant subspaces of dimension 2 if and only if $$a_1a_2\neq d_1d_2\text{ and }a_1d_2\neq a_2d_1.$$
\end{prop}
\begin{proof}
The subspaces of dimension 2 that are possibly invariant are those spanned by the following subsets of vectors: $\{e_1,he_2+je_3\},~\{e_4,re_2+se_3\},~\{e_2,ke_1+me_4\},~\{e_3,pe_1+qe_4\},~\{e_1+xe_4,e_2+ye_3\}$, where $\{h,j,r,s,k,m,p,q,x,y\}\subset\mathbb{C}$.\\
$\rho(\sigma_1)(e_1)=\left(\begin{matrix}
a_1a_2\\
a_1c_2\\
c_1a_2\\
c_1c_2
\end{matrix}\right)\not\in\langle e_1,he_2+je_3\rangle$ because $c_1c_2\neq0$.
So $\langle e_1,he_2+je_3\rangle$ are not invariant subspaces for all $h,j\in\mathbb{C}$.\\
$\rho(\sigma_1)(e_4)=\left(\begin{matrix}
1\\
d_2\\
d_1\\
d_1d_2
\end{matrix}\right)\not\in\langle e_4,re_2+se_3\rangle$.
So, $\langle e_4,re_2+se_3\rangle$ are not invariant subspaces for all $r,s\in\mathbb{C}$.\\
$\rho(\sigma_1)(e_2)=\left(\begin{matrix}
a_1\\
a_1d_2\\
c_1\\
c_1d_2
\end{matrix}\right)\not\in\langle e_2,ke_1+me_4\rangle$ because $c_1\neq0$.
So, the subspaces $\langle e_2,ke_1+me_4\rangle$ are not invariant for all $k,m\in\mathbb{C}$.\\
$\rho(\sigma_1)(e_3)=\left(\begin{matrix}
a_2\\
c_2\\
d_1a_2\\
d_1c_2
\end{matrix}\right)\not\in\langle e_3,pe_1+qe_4\rangle$ because $c_2\neq0$.
So, the subspaces $\langle e_3,pe_1+qe_4\rangle$ are not invariant for all $p,q\in\mathbb{C}$.\\\\
Now consider the subspaces of the form $\langle e_1+xe_4,e_2+ye_3\rangle$ where $x,y\in\mathbb{C}^*.$\\
Put $u=e_1+xe_4$ and $v=e_2+ye_3$ then \\
$\rho(\sigma_j)(u)=\left(\begin{matrix}
a_1a_2+x\\
\pm a_1c_2\pm d_2x\\
\pm c_1a_2\pm d_1x\\
c_1c_2+d_1d_2x
\end{matrix}\right)$ 
and 
$\rho(\sigma_j)(v)=\left(\begin{matrix}
\pm a_1\pm a_2y\\
a_1d_2+c_2y\\
c_1+d_1a_2y\\
\pm c_1d_2\pm d_1c_2y
\end{matrix}\right).$\\\\
The subspace $\langle u,v\rangle$ is invariant if and only if $\rho(\sigma_j)(u)=\alpha u+\alpha'v$ and $\rho(\sigma_j)(v)=\beta u+\beta'v$ for some $\alpha, \alpha', \beta, \beta'\in\mathbb{C}$. \\\\
This is equivalent to 
$\left\{\begin{matrix}
a_1a_2+x=\alpha\\
\pm(a_1c_2+d_2x)=\alpha'\\
\pm(c_1a_2+d_1x)=\alpha'y\\
c_1c_2+d_1d_2x=\alpha x
\end{matrix}\right.$
and 
$\left\{\begin{matrix}
\pm(a_1+a_2y)=\beta\\
a_1d_2+c_2y=\beta'\\
c_1+d_1a_2y=\beta'y\\
\pm(c_1d_2+d_1c_2y)=\beta x
\end{matrix}\right..$ \\
By eliminating $\alpha,\alpha', \beta$ and $\beta'$, we get the following system.
\begin{subequations}
\begin{align}
c_1c_2+d_1d_2x=a_1a_2x+x^2\\
c_1a_2+d_1x=a_1c_2y+d_2xy\\
c_1d_2+d_1c_2y=a_1x+a_2xy\\
c_1+d_1a_2y=a_1d_2y+c_2y^2
\end{align}
\end{subequations}
The equations (3.1a) and (3.1d) lead to four solutions:
$$x=\dfrac{-a_1a_2+d_1d_2\pm\sqrt{4c_1c_+(a_1a_2-d_1d_2)^2}}{2};\quad y=\dfrac{a_2d_1-a_1d_2\pm\sqrt{4c_1c_+(a_2d_1-a_1d_2)^2}}{2c_2}$$
Substitute each of these solutions into the equations (3.1b) and (3.1c) then, by using Mathematica, we get the following relations:\\
$(a_2=d_2),~~~(a_1=d_1),~~~(d_1=e^{\pm\pi i/3}a_1), ~~~ (d_1d_2=a_1a_2),~~~ (d_1a_2=a_1d_2),~~~ (a_1=a_2=0),~~~(a_1=d_2=0),~~~(d_1=a_2=0),~~~(d_1=d_2=0),~~~(d_1=a_1,a_2=0)~~~(d_1=a_1,d_2=0),~~~(d_1=e^{\pm\pi i/3}a_1,a_2=0)$ or $(d_1=e^{\pm\pi i/3}a_1,d_2=0)$\\
But we have $a_1\neq d_1,~~ a_2\neq d_2,~~ a_1a_2\neq d_1d_2,~~ a_1d_2\neq a_2d_1~$ and $a_1d_1a_2d_2\neq0$. So, the only 2 relations left are  $d_1=e^{\pm\pi i/3}a_1$ which lead to $c_1=a_1^2-a_1d_1+d_1^2=0$. This contradicts the fact that $c_1\neq 0$.\\
Therefore $\langle u,v\rangle$ is not invariant under $\rho(\sigma_j)$ for $1\leqslant j\leqslant4$.

\end{proof}

\begin{prop}
$\rho:{\mathcal{NB}}_4\to GL(4,\mathbb{C})$ has no non-trivial proper invariant subspaces of dimension 3 if and only if 
$$a_1a_2\neq d_1d_2\text{ and }a_1d_2\neq a_2d_1.$$
\end{prop}
\begin{proof}
The subspaces of dimension 3 that are possibly invariant are those spanned by the following subsets of vectors: $\{e_1,e_4,e_3\},~\{e_1,e_4,e_2\},~\{e_1,e_2,e_3\},~\{e_4,e_2,e_3\},~\{e_1,e_4,e_2+xe_3\},~\{e_1+ye_4,e_2,e_3\}$, where $x,y\in\mathbb{C}^*$.\\
$\rho(\sigma_1)(e_3)=\left(\begin{matrix}
a_2\\
c_2\\
d_1a_2\\
d_1c_2
\end{matrix}\right)\not\in\langle e_1,e_4,e_3\rangle$ since $c_2\neq0$. Thus, $\langle e_1,e_4,e_3\rangle$ is not invariant.\\
$\rho(\sigma_1)(e_2)=\left(\begin{matrix}
a_1\\
a_1d_2\\
c_1\\
c_1d_2
\end{matrix}\right)\not\in\langle e_1,e_4,e_2\rangle$ since $c_1\neq0$. So, $\langle e_1,e_4,e_2\rangle$ is not invariant \\
$\rho(\sigma_1)(e_1)=\left(\begin{matrix}
a_1a_2\\
a_1c_2\\
c_1a_2\\
c_1c_2
\end{matrix}\right)\not\in\langle e_1,e_2,e_3\rangle$ since  $c_1c_2\neq0$.
So, $\langle e_1,e_2,e_3\rangle$ is not invariant.\\
$\rho(\sigma_1)(e_4)=\left(\begin{matrix}
1\\
d_2\\
d_1\\
d_1d_2
\end{matrix}\right)\not\in\langle e_4,e_2,e_3\rangle$.
So, $\langle e_4,e_2,e_3\rangle$ is not invariant.\\\\
Now, consider the subspace $\langle e_1,e_4,e_2+xe_3\rangle$ with $x\neq0$.\\
Note that $\rho(\sigma_j)(e_1)=\left(\begin{matrix}
a_1a_2\\
\pm a_1c_2\\
\pm c_1a_2\\
c_1c_2
\end{matrix}\right)$,~$\rho(\sigma_j)(e_4)=\left(\begin{matrix}
1\\
\pm d_2\\
\pm d_1\\
d_1d_2
\end{matrix}\right)$ and $\rho(\sigma_j)(e_2+xe_3)=\left(\begin{matrix}
\pm a_1\pm a_2x\\
a_1d_2+c_2x\\
c_1+d_1a_2x\\
\pm c_1d_2\pm d_1c_2x
\end{matrix}\right)$.\\\\
Assume that the subspace $\langle e_1,e_4,e_2+xe_3\rangle$ is invariant. It follows that
$\{\rho(\sigma_1)(e_1),\rho(\sigma_1)(e_4), \rho(\sigma_1)(e_2+xe_3)\}\subset\langle e_1,e_4,e_2+xe_3\rangle$. This is equivalent to the following system.
\begin{subequations}
\begin{align}
&c_1a_2=a_1c_2x \\
&d_1=d_2x \\
&c_2x^2+(a_1d_2-d_1a_2)x-c_1=0
\end{align}
\end{subequations}
We determine conditions for which the system above is consistent. \\
From equations (3.2a) and (3.2b), we get
$x=\dfrac{c_1a_2}{a_1c_2}=\dfrac{d_1}{d_2}$ then substitute it into the equation (3.2c) to obtain $c_2\dfrac{d_1^2}{d_2^2}+a_1d_1-\dfrac{d_1^2a_2}{d_2}-c_1=0.$\\\\
$\Longrightarrow a_2d_2c_1=a_1d_1c_2\text{ and }c_2d_1^2-a_2d_2d_1^2=d_2^2c_1-d_2^2a_1d_1$\\\\
$\Longrightarrow  a_2d_2(a_1^2-a_1d_1+d_1^2)=a_1d_1(a_2^2-a_2d_2+d_2^2) ~ \text{ and } ~ d_1^2(c_2-a_2d_2)=d_2^2(c_1-a_1d_1)$\\\\
$\Longrightarrow  a_2d_2(a_1^2-2a_1d_1+d_1^2)=a_1d_1(a_2^2-2a_2d_2+d_2^2) ~ \text{ and } ~ d_1^2(a_2-d_2)^2=d_2^2(a_1-d_1)^2$\\\\
$\Longrightarrow  \dfrac{a_2d_2}{a_1d_1}=\dfrac{(a_2-d_2)^2}{(a_1-d_1)^2}=\dfrac{d_2^2}{d_1^2}$\\\\
$\Longrightarrow \dfrac{a_2}{a_1}=\dfrac{d_2}{d_1}$\\\\
$\Longrightarrow a_2d_1=a_1d_2$.\\\\
This gives a contradiction.\\
Now, consider the subspace $\langle e_1+ye_4,e_2,e_3\rangle$ with $y\neq0$. Note that \\
$\rho(\sigma_j)(e_2)=\left(\begin{matrix}
\pm a_1\\
a_1d_2\\
c_1\\
\pm c_1d_2
\end{matrix}\right)$,~
$\rho(\sigma_j)(e_3)=\left(\begin{matrix}
\pm a_2\\
c_2\\
d_1a_2\\
\pm d_1c_2
\end{matrix}\right)$~ and ~$\rho(\sigma_j)(e_1+ye_4)=\left(\begin{matrix}
a_1a_2+y\\
\pm a_1c_2\pm d_2y\\
\pm c_1a_2\pm d_1y\\
c_1c_2+d_1d_2y
\end{matrix}\right)$.\\\\
Assume that the subspace $\langle e_1+ye_4,e_2,e_3\rangle$ is invariant. It follows that the vectors $\rho(\sigma_i)(e_2),~~\rho(\sigma_i)(e_3)$ and $\rho(\sigma_i)(e_1+ye_4)$ are linear combinations of $e_1+ye_4,~e_2,$ and $e_3$. \\
$\Longrightarrow\dfrac{c_1d_2}{a_1}=\dfrac{d_1c_2}{a_2}=y$ and $y^2+(a_1a_2-d_1d_2)y-c_1c_2=0$\\
\\$\Longrightarrow a_2d_2c_1=a_1d_1c_2 ~ \text{ and } ~ \dfrac{c_1^2d_2^2}{a_1^2}+a_2d_2c_1-\dfrac{c_1d_1d_2^2}{a_1}-c_1c_2=0$\\\\
$\Longrightarrow a_2d_2(a_1^2-a_1d_1+d_1^2)=a_1d_1(a_2^2-a_2d_2+d_2^2) ~ \text{ and } ~ c_1d_2^2+a_2d_2a^2_1-d_1d_2^2a_1-c_2a_1^2=0$\\\\
$\Longrightarrow a_2d_2(a_1^2-2a_1d_1+d_1^2)=a_1d_1(a_2^2-2a_2d_2+d_2^2) ~ \text{ and } ~ d_2^2(c_1-d_1a_1)=a_1^2(c_2-a_2d_2)$\\\\
$\Longrightarrow a_2d_2(a_1-d_1)^2=a_1d_1(a_2-d_2)^2 ~ \text{ and } ~ d_2^2(a_1-d_1)^2=a_1^2(a_2-d_2)^2$\\\\
$\Longrightarrow\dfrac{a_2d_2}{a_1d_1}=\dfrac{d_2^2}{a_1^2}=\dfrac{(a_2-d_2)^2}{(a_1-d_1)^2}$\\
$\Longrightarrow \dfrac{a_2}{d_1}=\dfrac{d_2}{a_1}$.\\\\
This gives a contradiction. 
\end{proof}
We thus get our main theorem.
\begin{thm}
The representation $\rho:\mathcal{NB}_4\to GL(4,\mathbb{C})$ is irreducible if and only if $a_1a_2\neq d_1d_2$ and $a_1d_2\neq a_2d_1$.
\end{thm}

\subsubsection{Case $T_1=T_2=\pm i$}
In the case $T_1=T_2=i$, the representation $\rho$ of $\mathcal{NB}_4$ is given by
$$\rho(\tau)=
t_1t_2\left(\begin{matrix}
-1&0&0&0\\
0&i&0&0\\
0&0&i&0\\
0&0&0&1
\end{matrix}\right),
~
\rho(\sigma_1)=\left(\begin{matrix}
d_1d_2&d_1&d_2&1\\
-d_1d^2_2&d_1d_2&-d^2_2&d_2\\
-d^2_1d_2&-d_1^2&d_1d_2&d_1\\
d^2_1d^2_2&-d^2_1d_2&-d_1d^2_2&d_1d_2
\end{matrix}\right),$$
$$\rho(\sigma_2)=\left(\begin{matrix}
d_1d_2& id_1& id_2&-1\\
 id_1d^2_2&d_1d_2&-d^2_2& id_2\\
 id^2_1d_2&-d_1^2&d_1d_2& id_1\\
-d^2_1d^2_2& id^2_1d_2& id_1d^2_2&d_1d_2
\end{matrix}\right),
~
\rho(\sigma_3)=\left(\begin{matrix}
d_1d_2&-d_1&-d_2&1\\
d_1d^2_2&d_1d_2&-d^2_2&-d_2\\
d^2_1d_2&-d_1^2&d_1d_2&-d_1\\
d^2_1d^2_2&d^2_1d_2&d_1d^2_2&d_1d_2
\end{matrix}\right),$$
$$\rho(\sigma_4)=\left(\begin{matrix}
d_1d_2&- id_1&- id_2&-1\\
 id_1d^2_2&d_1d_2&-d^2_2&- id_2\\
 id^2_1d_2&-d_1^2&d_1d_2&- id_1\\
-d^2_1d^2_2&- id^2_1d_2&- id_1d^2_2&d_1d_2
\end{matrix}\right).$$
Likewise, we write the matrices for $T_1=T_2=-i.$
\begin{prop}
The representation $\rho$ is reducible.
\end{prop}
\begin{proof}
Let $S$ be the subspace spanned by the vector $v=d_2e_2-d_1e_3$. Note that $e_2$ and $e_3$ are eigen-vectors of $\rho(\tau )$ corresponding the same eigenvalue. So, $S$ is invariant under $\rho(\tau)$. It is easy to show that $\rho(\sigma_j)(v)=2d_1d_2v$ for $j=1,2,3,4$. This completes the proof.
\end{proof}
\subsubsection{Case $T_1=-T_2=\pm i$}
In the case  $T_1=-T_2=i$, the representation $\rho$ of $\mathcal{NB}_4$ is given by
 $$\rho(\tau)=
t_1t_2\left(\begin{matrix}
1&0&0&0\\
0& i&0&0\\
0&0&- i&0\\
0&0&0&1
\end{matrix}\right),
~
\rho(\sigma_1)=\left(\begin{matrix}
d_1d_2&d_1&d_2&1\\
-d_1d^2_2&d_1d_2&-d^2_2&d_2\\
-d^2_1d_2&-d_1^2&d_1d_2&d_1\\
d^2_1d^2_2&-d^2_1d_2&-d_1d^2_2&d_1d_2
\end{matrix}\right),$$
$$\rho(\sigma_2)=\left(\begin{matrix}
d_1d_2&- id_1& id_2&1\\
- id_1d^2_2&d_1d_2&d^2_2& id_2\\
 id^2_1d_2&d_1^2&d_1d_2&- id_1\\
d^2_1d^2_2& id^2_1d_2&- id_1d^2_2&d_1d_2
\end{matrix}\right),
~
\rho(\sigma_3)=\left(\begin{matrix}
d_1d_2&-d_1&-d_2&1\\
d_1d^2_2&d_1d_2&-d^2_2&-d_2\\
d^2_1d_2&-d_1^2&d_1d_2&-d_1\\
d^2_1d^2_2&d^2_1d_2&d_1d^2_2&d_1d_2
\end{matrix}\right),$$

$$\rho(\sigma_4)=\left(\begin{matrix}
d_1d_2& id_1&- id_2&1\\
 id_1d^2_2&d_1d_2&d^2_2&- id_2\\
- id^2_1d_2&d_1^2&d_1d_2& id_1\\
d^2_1d^2_2&- id^2_1d_2& id_1d^2_2&d_1d_2
\end{matrix}\right).$$
Likewise, we write the matrices for $T_1=-T_2=-i.$
\begin{prop}
The representation $\rho$ is reducible.
\end{prop}
\begin{proof}
Let $S$ be the subspace spanned by the vector $v=e_1+d_1d_2e_4$. The subspace $S$ is invariant under $\rho(\tau)$ because, $e_1$ and $e_4$ are eigenvectors of $\rho(\tau)$ corresponding to the same eigenvalue.
By direct calculations we get $\rho(\sigma_j)(v)=2d_1d_2v\in S$ for $j=1,2,3,4$. Hence $S$ is invariant under $\rho(\sigma_j)$.
\end{proof}

\subsubsection{Case $T_1=-1,T_2=\pm i$}
Without loss of generality, assume that $T_1=-1$ and $T_2=i$. Then we have $T_1T_2=-i$ and 
 $$\rho(\tau)=
t_1t_2\left(\begin{matrix}
- i&0&0&0\\
0&-1&0&0\\
0&0& i&0\\
0&0&0&1
\end{matrix}\right),
~
\rho(\sigma_1)=\left(\begin{matrix}
a_1d_2&a_1&d_2&1\\
-a_1d^2_2&a_1d_2&-d^2_2&d_2\\
c_1d_2&c_1&d_1d_2&d_1\\
-c_1d^2_2&c_1d_2&-d_1d^2_2&d_1d_2
\end{matrix}\right) $$
Likewise, we write the matrices for $T_1=-1, T_2=-i.$\\
Note that the vectors $e_j$ ($j=1,2,3,4$) are eigenvectors of $\rho(\tau)$ corresponding to four distinct eigenvalues (each of multiplicity 1).
\begin{prop}
The representation $\rho$ is irreducible.
\end{prop}
\begin{proof}
We have 
$\rho(\sigma_1)(e_1)=\left(\begin{matrix}
a_1d_2\\
 -a_1d_2^2\\ 
 c_1d_2\\ 
 -c_1d_2^2
\end{matrix}\right)$,
$\rho(\sigma_1)(e_2)=\left(\begin{matrix}
a_1\\
 a_1d_2\\ 
 c_1\\ 
 c_1d_2
\end{matrix}\right)$,
$\rho(\sigma_1)(e_3)=\left(\begin{matrix}
d_2\\
-d_2^2\\ 
 d_1d_2\\ 
 -d_1d_2^2
\end{matrix}\right)$ and 
$\rho(\sigma_1)(e_4)=\left(\begin{matrix}
1\\
d_2\\ 
d_1\\ 
d_1d_2
\end{matrix}\right)$. 
Note that $a_1\neq0$, $d_1\neq0$, $c_1\neq0$ and $d_2\neq0$. This implies that none of the proper subspaces spanned by the vectors $e_j$ are invariant under $\rho(\sigma_1)$. Likewise, we show that there are no invariant subspaces of dimensions 2 and 3. Hence $\rho$ is irreducible. 
\end{proof}
\medskip
\subsection{Representation of $\mathcal{NB}_3$}

Consider two irreducible representations of dimesion 2 of the necklace braid group $\mathcal{NB}_3$ $\rho_1=\rho(T_1,t_1,\omega_1,c_1,d_1)$ and $\rho_2=\rho(T_2,t_2,\omega_2,c_2,d_2)$ (see Proposition 2.1).\\\\
Recall that $T_1=e^{\pm i2\pi/3},T_2=e^{\pm i2\pi/3},\omega_1=e^{\pm i\pi/3}$ and $\omega_2=e^{\pm i\pi/3}$.\\\\
Let $\rho$ be the representation of $\mathcal{NB}_3$, which is defined in Definition 3.1. We have
$$
\rho(\tau)=
t_1t_2\left(\begin{matrix}
T_1T_2&0&0&0\\
0&T_1&0&0\\
0&0&T_2&0\\
0&0&0&1
\end{matrix}\right)
$$
and 
$$\rho(\sigma_j)=
\left(\begin{matrix}
\omega_1\omega_2d_1d_2&\omega_1d_1T_2^{j-1}&\omega_2d_2T_1^{j-1}&(T_1T_2)^{j-1}\\
\omega_1d_1c_2T_2^{1-j}&\omega_1d_1d_2&c_2T_1^{j-1}T_2^{1-j}&d_2T_1^{j-1}\\
\omega_2c_1d_2T_1^{1-j}&c_1T_1^{1-j}T_2^{j-1}&\omega_2d_2d_1&d_1T_2^{j-1}\\
c_1c_2(T_1T_2)^{1-j}&c_1d_2T_1^{1-j}&d_1c_2T_2^{1-j}&d_1d_2
\end{matrix}\right).$$
 for $j=1,2,3$. We have two cases (1) $T_1=e^{2\pi i/3},T_2=e^{-2\pi i/3}$ and (2) $T_1=T_2=e^{\pm 2\pi i/3}$.\\

We determine a necessary and sufficient condition under which $\rho$ is an irreducible representation of $\mathcal{NB}_3$ of dimension 4.\\

In what follows, suppose that $d_1d_2\neq0$.
\subsubsection{Case $T_1=e^{2\pi i/3}, T_2=e^{-2\pi i/3}$}
In this case, we have  $$\rho(\tau)=
t_1t_2\left(\begin{matrix}
1&0&0&0\\
0&e^{2\pi i/3}&0&0\\
0&0&e^{-2\pi  i/3}&0\\
0&0&0&1
\end{matrix}\right)
.$$ Then the eigenvalues of $\rho(\tau)$ are $\lambda_1=t_1t_2$ (of multiplicity 2) and $\lambda_2=e^{2\pi i/3}t_1t_2$  and $\lambda_3=e^{-2\pi i/3}t_1t_2$. The corresponding eigenvectors are $x_1e_1+x_4e_4$ for $\lambda_1$ and $e_2$ and $e_3$ for $\lambda_2$ and $\lambda_3$ respectively, where $\{x_1,x_4\}\subset\mathbb{C}^*$.\\
\begin{lem}
If $T_1=e^{2\pi i/3}$, $T_2=e^{-2\pi i/3},\omega_1=e^{\pm i\pi/3}$ and $\omega_2=e^{\pm i\pi/3}$ then the representation $\rho$ of $\mathcal{NB}_3$ has no invariant proper subspaces of dimension 1 if and only if  $$\omega_1\omega_2\neq1 \text{ or }\omega_1d_1^2c_2\neq\omega_2d_2^2c_1.$$
\end{lem}
\begin{proof}
By direct calculations, we check that each of the subspaces spanned by the vectors $e_j$ ($j=1,2,3$) are not invariant under $\rho(\sigma_1)$.\\
It remains to consider the subspace $S$ spanned by the vector $v=e_1+xe_4$ with $x\in\mathbb{C}^*$.
Since $T_1T_2=1$, it follows that $$\rho(\sigma_j)(v)=\left(\begin{matrix}
\omega_1\omega_2d_1d_2+x\\
\omega_1d_1c_2T_2^{1-j}+d_2T_1^{j-1}x\\
\omega_2c_1d_2T_1^{1-j}+d_1T_2^{j-1}x\\
c_1c_2+d_1d_2x
\end{matrix}\right)=
\left(\begin{matrix}
\omega_1\omega_2d_1d_2+x\\
(\omega_1d_1c_2+d_2x)T_1^{j-1}\\
(\omega_2c_1d_2+d_1x)T_2^{j-1}\\
c_1c_2+d_1d_2x
\end{matrix}\right).$$ 
So
\begin{align*}
\rho(\sigma_j)(v)\in S & \Leftrightarrow  \rho(\sigma_j)(v)=\lambda v \text{ for some }\lambda\in\mathbb{C} \\ &
\Leftrightarrow x=\frac{-\omega_1d_1c_2}{d_2}=\frac{-\omega_2c_1d_2}{d_1} \text{ and }  x^2+(\omega_1\omega_2-1)d_1d_2x-c_1c_2=0\\ &\Leftrightarrow \omega_1d_1^2c_2=\omega_2d_2^2c_1\text{ and }
(\omega_1\omega_2-1)(c_2-\omega_2d_2^2)=0\\ & \Leftrightarrow
\omega_1d_1^2c_2=\omega_2d_2^2c_1\text{ and }
\omega_1\omega_2-1=0.
\end{align*}
Therefore the subspace $S$, that is spanned by $v=e_1-\frac{\omega_1d_1c_2}{d_2}e_4$, is  invariant under $\rho(\sigma_j)$ if $\omega_1\omega_2=1$ and $\omega_1d_1^2c_2= \omega_2d_2^2c_1$. This gives a contradiction.
\end{proof}
\begin{lem}
If $T_1=e^{2\pi i/3}$, $T_2=e^{-2\pi i/3},\omega_1=e^{\pm i\pi/3}$ and $\omega_2=e^{\pm i\pi/3}$ then the representation $\rho$ of $\mathcal{NB}_3$ has no invariant proper subspaces of dimension 2.
\end{lem}
\begin{proof}
The possible  two dimensional invariant subspaces candidates to study are the following:
\begin{itemize}
\item[1.] $\langle e_i,e_j\rangle$ for $i\neq j$
\item[2.] $\langle e_2,e_1+xe_4\rangle$ for $x\neq0$
\item[3.] $\langle e_3,e_1+ye_4\rangle$ for $y\neq0$
\end{itemize} 
Since $c_1\neq0, c_2\neq0$, $d_1\neq0$ and $d_2\neq0$, it follows that none of the subspaces mentioned above is invariant under $\rho(\sigma_1)$.\\
\end{proof}
\begin{lem}
If $T_1=e^{2\pi i/3}$, $T_2=e^{-2\pi i/3}, \omega_1=e^{\pm i\pi/3}$ and $\omega_2=e^{\pm i\pi/3}$ then the representation $\rho$ of $\mathcal{NB}_3$ has no invariant proper subspaces of dimension 3 if and only if  $$\omega_1\omega_2\neq1 \text{ or }\omega_1d_1^2c_2\neq\omega_2d_2^2c_1.$$
\end{lem}
\begin{proof}
The possible invariant three dimensional invariant subspaces to study are:
 \begin{itemize}
\item[1.] $\langle e_i,e_j,e_k\rangle$ where $i, j, k$ are pairwise distinct
\item[2.] $\langle e_2,e_3,e_1+xe_4\rangle$ for $x\neq0$
\end{itemize}
By direct computations, it is easy to check that all the subspaces in case 1 are not invariant.\\
In case 2, we let $S$ be a subspace of the form $\langle e_2,e_3,e_1+xe_4\rangle$ with $x\neq0$.\\
$$\rho(\sigma_j)(e_2)=\left(\begin{matrix}
\omega_1d_1T_2^{j-1}\\
 \omega_1d_1d_2\\ 
 c_1T_1^{1-j}T_2^{j-1}\\ 
 c_1d_2T_1^{1-j}
\end{matrix}\right),
~\rho(\sigma_j)(e_3)=\left(\begin{matrix}
\omega_2d_2T_1^{j-1}\\
 c_2T_1^{j-1}T_2^{1-j}\\ 
 \omega_2d_1d_2\\ 
 d_1c_2T_2^{1-j}
\end{matrix}\right)$$
and $$\rho(\sigma_j)(e_1+xe_4)=\left(\begin{matrix}
\omega_1\omega_2d_1d_2+x\\
 \omega_1d_1c_2T_2^{1-j}+d_2T_1^{j-1}x\\ 
 \omega_2c_1d_2T_1^{1-j}+d_1T_2^{j-1}x\\ 
 c_1c_2+d_1d_2x
\end{matrix}\right),$$
Therefore
$$\rho(\sigma_j)(e_2)\in S\Leftrightarrow x=\dfrac{c_1d_2}{\omega_1d_1} $$ and 
$$\rho(\sigma_j)(e_3)\in S\Leftrightarrow x=\dfrac{d_1c_2}{\omega_2d_2},$$
\begin{align*}
\rho(\sigma_j)(e_1+xe_4)\in S&\Leftrightarrow \left(\begin{matrix}
\omega_1\omega_2d_1d_2+x\\
 \omega_1d_1c_2T_2^{1-j}+d_2T_1^{j-1}x\\ 
 \omega_2c_1d_2T_1^{1-j}+d_1T_2^{j-1}x\\ 
 c_1c_2+d_1d_2x
\end{matrix}\right)=
\left(\begin{matrix}
\lambda\\
 \alpha\\ 
 \beta\\ 
 \lambda x
\end{matrix}\right)\text{ for some } \alpha, \beta, \gamma\in\mathbb{C}.
\end{align*}
This is equivalent to $$ x^2+(\omega_1\omega_2d_1d_2-d_1d_2)x-c_1c_2=0.$$

Therefore, $S$ is invariant under $\rho(\sigma_j)$ ($j=1,2,3$) if and only if
\begin{equation}
x=\dfrac{c_1d_2}{\omega_1d_1}=\dfrac{c_2d_1}{\omega_2d_2}
\end{equation}
 and 
\begin{equation}
 c_1^2d_2^2+\omega_1d_1^2c_1d_2^2(\omega_1\omega_2-1)-\omega_1d_1^2c_1c_2=0.
 \end{equation}
The equations (3.3) and (3.4) imply that
\begin{equation*}
\omega_1d_1^2c_2=\omega_2d_2^2c_1
\end{equation*}
 and 
\begin{equation*}
(\omega_1\omega_2-1)(\omega_2d_2^2-c_2)=0.
\end{equation*}
Since $\omega_2d_2^2\neq c_2$, it follows that $$\omega_1\omega_2=1$$
Therefore the subspace $S$, which is spanned by $v=e_1-\frac{c_1d_2}{\omega_1d_1}e_4$, is  invariant under $\rho(\sigma_j)$ if $\omega_1\omega_2=1$ and $\omega_1d_1^2c_2= \omega_2d_2^2c_1$. This gives a contradiction.
\end{proof}
Now, we have the following proposition.
\begin{prop}
If $T_1=e^{i2\pi/3}$, $T_2=e^{-i2\pi/3}, \omega_1=e^{\pm i\pi/3}$ and $\omega_2=e^{\pm i\pi/3}$ then $\rho$ is irreducible representation of $\mathcal{NB}_3$ if and only if  $\omega_1\omega_2\neq 1$ or $\omega_2d_2^2c_1\neq\omega_1d_1^2c_2$.
\end{prop}

\subsubsection{Case  ($T_1=T_2=e^{2\pi i/3}$) and ($T_1=T_2=e^{-2\pi i/3}$)}

In this case, we have $T_1=T_2$. Then
$$
\rho(\tau)=
t_1t_2\left(\begin{matrix}
e^{\mp i2\pi/3}&0&0&0\\
0&e^{\pm i2\pi/3}&0&0\\
0&0&e^{\pm i2\pi/3}&0\\
0&0&0&1
\end{matrix}\right)
$$ and
$$  \rho(\sigma_j)=
\left(\begin{matrix}
\omega_1\omega_2d_1d_2&\omega_1d_1T_1^{j-1}&\omega_2d_2T_1^{j-1}&(T_1)^{2j-2}\\
\omega_1d_1c_2T_1^{1-j}&\omega_1d_1d_2&c_2&d_2T_1^{j-1}\\
\omega_2c_1d_2T_1^{1-j}&c_1&\omega_2d_2d_1&d_1T_1^{j-1}\\
c_1c_2(T_1)^{2-2j}&c_1d_2T_1^{1-j}&d_1c_2T_1^{1-j}&d_1d_2
\end{matrix}\right)$$
for $j=1,2,3$. The eigenvalues of $\rho(\tau)$ are then $\lambda_2=t_1t_2e^{\pm i2\pi/3}$ (of multiplicity 2) and $\lambda_1=e^{\mp i2\pi/3}t_1t_2$  and $\lambda_3=t_1t_2$. The corresponding eigenvectors are $xe_2+ye_3$ for $\lambda_2$ and $e_1$ and $e_4$ for $\lambda_1$ and $\lambda_3$ respectively.\\

\begin{lem}
If $T_1=T_2=e^{\pm2\pi i/3},\omega_1=e^{\pm i\pi/3}$ and $\omega_2=e^{\pm i\pi/3}$ then the representation $\rho$ of $\mathcal{NB}_3$ has no invariant proper subspaces of dimension 1 if and only if  $\omega_1\neq\omega_2$ or $\omega_1d_1^2c_2\neq\omega_2d_2^2c_1$.
\end{lem}
\begin{proof}
By direct calculations, we check that each of the subspaces spanned by the vectors $e_j$ are not invariant under $\rho(\sigma_1)$.\\
It remains to consider the subspace $S$ spanned by the vector $v=e_2+xe_3$ with $x\in\mathbb{C}^*$.
Since $T_1=T_2$, it follows that $$\rho(\sigma_j)(v)=\left(\begin{matrix}
\omega_1d_1T_1^{j-1}+\omega_2d_2T_1^{j-1}x\\
\omega_1d_1d_2+c_2x\\
c_1+\omega_2d_2d_1x\\
c_1d_2T_1^{1-j}+d_1c_2T_1^{1-j}x
\end{matrix}\right)=
\left(\begin{matrix}
(\omega_1d_1+\omega_2d_2x)T_1^{j-1}\\
\omega_1d_1d_2+c_2x\\
c_1+\omega_2d_2d_1x\\
(c_1d_2+d_1c_2x)T_1^{1-j}
\end{matrix}\right).$$ So 
\begin{align*}
\rho(\sigma_j)(v)\in S & \Leftrightarrow  \rho(\sigma_j)(v)=\lambda v \text{ for some }\lambda\in\mathbb{C} \\ &
\Leftrightarrow x=\frac{-\omega_1d_1}{\omega_2d_2}=\frac{-c_1d_2}{d_1c_2} \text{ and } c_2 x^2+(\omega_1-\omega_2)d_1d_2x-c_1=0\\ &\Leftrightarrow \omega_1d_1^2c_2=\omega_2d_2^2c_1\text{ and }
c_2\omega_1^2d_1^2+(\omega_2-\omega_1)d_1^2d_2^2\omega_1\omega_2-c_1\omega_2^2d_2^2=0\\ &\Leftrightarrow \omega_1d_1^2c_2=\omega_2d_2^2c_1\text{ and }
(\omega_1-\omega_2)(c_2-\omega_2d_2^2)=0\\& \Leftrightarrow
\omega_1d_1^2c_2=\omega_2d_2^2c_1\text{ and }
\omega_1=\omega_2\text{ because } c_2\neq\omega_2d_2^2
\end{align*}
Therefore the subspace $S$, which is spanned by $v=e_2-\frac{\omega_1d_1}{\omega_2d_2}e_3$, is  invariant under $\rho(\sigma_j)$ if $\omega_1=\omega_2$ and $\omega_1d_1^2c_2= \omega_2d_2^2c_1$.
This gives a contradiction.
\end{proof}
\begin{lem}
If $T_1=T_2=e^{\pm2\pi i/3}$ and $\omega_1=e^{\pm i\pi/3}$ and $\omega_2=e^{\pm i\pi/3}$  then the representation $\rho$ of $\mathcal{NB}_3$ has no invariant proper subspaces of dimension 2.
\end{lem}
\begin{proof}
The possible  two dimensional invariant subspaces candidates to study are the following:
\begin{itemize}
\item[1.] $\langle e_i,e_j\rangle$ for $i\neq j$
\item[2.] $\langle e_1,e_2+xe_3\rangle$ for $x\neq0$
\item[3.] $\langle e_4,e_2+ye_3\rangle$ for $y\neq0$
\end{itemize} 
Since $c_1\neq0, c_2\neq0$, $d_1\neq0$ and $d_2\neq0$, it follows that none of the subspaces mentioned above is invariant under $\rho(\sigma_1).$
\end{proof}
\begin{lem}
If $T_1=T_2=e^{\pm2\pi i/3}$, $\omega_1=e^{\pm i\pi/3}$ and $\omega_2=e^{\pm i\pi/3}$ then the representation $\rho$ of $\mathcal{NB}_3$ has no invariant proper subspaces of dimension 3 if and only if  $\omega_1\neq\omega_2$ or $\omega_1d_1^2c_2\neq\omega_2d_2^2c_1$.
\end{lem}
\begin{proof}
The possible invariant three dimensional invariant subspaces to study are:
 \begin{itemize}
\item[1.] $\langle e_i,e_j,e_k\rangle$ where $i, j, k$ are pairwise distinct
\item[2.] $\langle e_1,e_4,e_2+xe_3\rangle$ for $x\neq0$
\end{itemize}
By direct computations, it is easy to check that all the subspaces in case 1 are not invariant under $\rho(\sigma_1)$.\\
In case 2, we let $S$ be a subspace of the form $\langle e_1,e_4,e_2+xe_3\rangle$ with $x\neq0$.\\
$$\rho(\sigma_j)(e_1)=\left(\begin{matrix}
\omega_1\omega_2d_1d_2\\
 \omega_1d_1c_2T_1^{1-j}\\ 
 \omega_2c_1d_2T_1^{1-j}T_2^{j-1}\\ 
 c_1c_2T_1^{2-2j}
\end{matrix}\right),
~\rho(\sigma_j)(e_4)=\left(\begin{matrix}
T_1^{2j-2}\\
 d_2T_1^{j-1}\\ 
 d_1T_1^{j-1}\\ 
 d_1d_2
\end{matrix}\right)$$
and $$\rho(\sigma_j)(e_2+xe_3)=\left(\begin{matrix}
\omega_1d_1T_1^{j-1}+\omega_2d_2T_1^{j-1}x\\
\omega_1d_1d_2+c_2x\\ 
c_1+\omega_2d_2d_1x\\ 
c_1d_2T_1^{1-j}+d_1c_2T_1^{1-j}x
\end{matrix}\right).$$
Therefore, we get
$$\rho(\sigma_j)(e_1)\in S\Leftrightarrow x=\dfrac{\omega_2c_1d_2}{\omega_1d_1c_2} $$ and 
$$\rho(\sigma_j)(e_4)\in S\Leftrightarrow x=\dfrac{d_1}{d_2},$$
\begin{align*}
\rho(\sigma_j)(e_2+xe_3)\in S&\Leftrightarrow \begin{pmatrix}
\omega_1d_1T_1^{j-1}+\omega_2d_2T_1^{j-1}x\\
\omega_1d_1d_2+c_2x\\ 
c_1+\omega_2d_2d_1x\\ 
c_1d_2T_1^{1-j}+d_1c_2T_1^{1-j}x
\end{pmatrix}=
\left(\begin{matrix}
\alpha\\
 \lambda\\ 
 \lambda x\\
 \beta
\end{matrix}\right)\text{ for some } \alpha, \beta, \gamma\in\mathbb{C}.
\end{align*}
This is equivalent to $$c_2x^2+(\omega_1-\omega_2)d_1d_2x-c_1=0.$$

Therefore $S$ is invariant under $\rho(\sigma_j)$ ($j=1,2,3$) if and only if:
\begin{equation}
x=\dfrac{d_1}{d_2}=\dfrac{\omega_2c_1d_2}{\omega_1d_1c_2}
\end{equation}
 and 
\begin{equation}
 c_2d_1^2+(\omega_1-\omega_2)d_1^2d_2^2-c_1d_2^2=0
 \end{equation}
The equations (3.5) and (3.6) imply that 
\begin{equation*}
\omega_1d_1^2c_2=\omega_2d_2^2c_1
\end{equation*}
 and 
\begin{equation*}
(\omega_1-\omega_2)(\omega_1d_1^2-c_1)=0.
\end{equation*}
Since $\omega_1d_1^2\neq c_1$, it follows that $\omega_1=\omega_2$, which gives a contradiction.
Therefore the subspace $S$, that is spanned by $v=e_2-\frac{d_1}{d_2}e_3$, is  invariant under $\rho(\sigma_j)$ if $\omega_1=\omega_2$ and $\omega_1d_1^2c_2= \omega_2d_2^2c_1$.
\end{proof}

\begin{prop}
If $T_1=T_2=e^{\pm 2\pi i/3}$, $\omega_1=e^{\pm i\pi/3}$ and $\omega_2=e^{\pm i\pi/3}$ then $\rho$ is an irreducible representation of $\mathcal{NB}_3$ if and only if $\omega_1\neq\omega_2$ or $\omega_2d_2^2c_1\neq\omega_1d_1^2c_2$.
\end{prop}
\begin{proof}
The proof follows directly from Lemma 3.12, Lemma 3.13 and Lemma 3.14.
\end{proof}
\begin{thm}
The representation $\rho$ of $\mathcal{NB}_3$ is irreducible if and only if one of the following conditions hold.
\begin{itemize}
\item[1.] $T_1=e^{2\pi i/3}, T_2=e^{-2\pi i/3}$ and $\omega_1\omega_2\neq1$
\item[2.] $T_1=e^{2\pi i/3}, T_2=e^{-2\pi i/3}$ and $\omega_1d_1^2c_2\neq\omega_2d_2^2c_1$
\item[3.] $T_1=T_2=e^{\pm2\pi i/3}$ and $\omega_1\neq\omega_2$
\item[4.] $T_1=T_2=e^{\pm2\pi i/3}$ and $\omega_1d_1^2c_2\neq\omega_2d_2^2c_1$
\end{itemize}
\end{thm}
\begin{proof}
The proof follows directly from Proposition 3.11 and Proposition 3.15 
\end{proof}
\medskip
\subsection{Representation of $\mathcal{NB}_2$}

In this section, we consider the representation $\rho$ given by Definition 3.1. In this case we have $T_1=T_2=-1$. A similar work to that done for $\mathcal{NB}_4$ (Theorem 3.4), we get the following theorem.
\begin{thm}
The representation $\rho$ is irreducible if and only if $a_1a_2\neq d_1d_2$ and $a_1d_2\neq a_2d_1$.
\end{thm}
\vspace{.4cm}
\section{Unitary representations of $\mathcal{NB}_n$ of dimension 2 ($n=2,3,4$)}
\bigskip
\begin{defn}
A square matrix $A$ is unitary relative to a matrix $M$ if $AMA^*=M$. Here $A^*$ is the conjugate transpose of $A$.
\end{defn}  
\begin{defn}
A representation $\rho:G\to GL(n,\mathbb{C})$ of a group $G$ is called unitary if $\rho(g)$ is unitary for every element $g\in G$.
\end{defn}
We consider the irreducible representations of $\mathcal{NB}_n$, given by Proposition 2.1. Then, we determine a necessary and sufficient condition under which these representations are unitary relative to some hermitian positive definite matrix.\\
Recall that
$$\rho(\tau)=\left(\begin{matrix}
Tt&0\\
0&t
\end{matrix}\right), ~ 
\rho(\sigma_j)=\begin{pmatrix}
a&T^{j-1}\\
cT^{1-j}&d
\end{pmatrix},$$
where $t$ is a $(2n)$-th root of unity and $T\in\{-1,i,-i,e^{i2\pi/3},e^{-i2\pi/3}\}$.
\begin{lem}
Let $\rho$ be a complex irreducible representation of $\mathcal{NB}_n$ ($n=2,3,4$), given by Proposition 2.1. If $\rho$ is unitary relative to a matrix $M=\left(\begin{matrix}
x&y\\
z&u
\end{matrix}\right)$ then $y=z=0$ and $$\left\{\begin{array}{l}
(|a|^2-1)x+u=0\\
|c|^2x+(|d|^2-1)u=0\\
a\bar{c}x+\bar{d}u=0
\end{array}\right..$$
\end{lem}
\begin{proof}
Since $\rho$ is unitary relative to a matrix 
$M=\left(\begin{matrix}
x&y\\
z&u
\end{matrix}\right)$, we have $\rho(\tau)M\rho(\tau)^*=M$.
Thus we get $$\left(\begin{matrix}
Tt&0\\
0&t
\end{matrix}\right)
\left(\begin{matrix}
x&y\\
z&u
\end{matrix}\right)
\left(\begin{matrix}
\bar{T}\bar{t}&0\\
0&\bar{t}
\end{matrix}\right)
=\left(\begin{matrix}
x&y\\
z&u
\end{matrix}\right).$$
 Then we get
$$\left(\begin{matrix}
x&Ty\\
\bar{T}z&u
\end{matrix}\right)=
\begin{pmatrix}
x&y\\
z&u
\end{pmatrix}.$$
It follows that $$Ty=y \text{ and } \bar{T}z=z.$$
Therefore $y=z=0$ because $T\neq 1$.\\

We also have $\rho(\sigma_j)M\rho(\sigma_j)^*=M$ ($j=1,2,3,4$).\\\\
Then we get
$$
\begin{pmatrix}
a&T^{j-1}\\
cT^{1-j}&d
\end{pmatrix}
\begin{pmatrix}
x&0\\
0&u
\end{pmatrix}
\begin{pmatrix}
\bar{a}&\bar{c}\bar{T}^{1-j}\\
\bar{T}^{j-1}&\bar{d}
\end{pmatrix}
=\begin{pmatrix}
x&0\\
0&u
\end{pmatrix}.$$
So we have
$$\begin{pmatrix}
|a|^2x+u&a\bar{c}\bar{T}^{1-j}x+\bar{d}T^{j-1}u\\
c\bar{a}T^{1-j}x+d\bar{T}^{j-1}u&|c|^2x+|d|^2u
\end{pmatrix}
=\begin{pmatrix}
x&0\\
0&u
\end{pmatrix}.
$$
Since $\bar{T}=T^{-1}$, it follows that
$$
\begin{pmatrix}
|a|^2x+u&(a\bar{c}x+\bar{d}u)T^{j-1}\\
(c\bar{a}x+du)T^{1-j}&|c|^2x+|d|^2u
\end{pmatrix}
=\begin{pmatrix}
x&0\\
0&u
\end{pmatrix}.$$
Then we get
$$\left\{\begin{array}{l}
(|a|^2-1)x+u=0\\
|c|^2x+(|d|^2-1)u=0\\
a\bar{c}x+\bar{d}u=0
\end{array}\right..
$$
\end{proof}
 \begin{prop}
Let $\rho:{\mathcal{NB}}_4\to GL(2,\mathbb{C})$ be the irreducible representation of ${\mathcal{NB}}_4$ given in Proposition 2.1.($T=-1$).
Then
$\rho$ is unitary relative to a hermitian positive definite matrix $M$ if and only if
$$|a|<1$$ and
$$a=d+e^{i\theta}, 
\left\{\begin{array}{ll}
d=\delta'-\frac{1}{2}i&\text{ if }\theta=\frac{\pi}{2}\text{ mod }2\pi \\
d=\delta'+\frac{1}{2}i&\text{ if }\theta=-\frac{\pi}{2}\text{ mod }2\pi \\
d=-\frac{1}{2}\sec\theta-(\tan\theta) \delta+i\delta,& \text{ if }\theta\neq\pm\frac{\pi}{2} \text{ mod }2\pi
\end{array}\right.,
$$
where $\theta,\delta',\delta\in\mathbb{R}$ 
 \end{prop}
 
 \begin{proof}
Recall that 
$$\rho(\tau)=\left(\begin{matrix}
-t&0\\
0&t
\end{matrix}\right)~\text{and }~
\rho(\sigma_j)=\left(\begin{matrix}
a&(-1)^{j-1}\\
c(-1)^{1-j}&d
\end{matrix}\right),$$
where $t$ is an eighth root of unity and $c=a^2-ad+d^2$ such that $a\neq d$ and $c\neq0$. Suppose that $\rho$ is unitary relative to matrix $M$. By Lemma 4.1, we have $M=\left(\begin{matrix}
 x&0\\
 0&u
 \end{matrix}\right)$ and
 $$
\left\{\begin{array}{l}
(|a|^2-1)x+u=0\\
|c|^2x+(|d|^2-1)u=0\\
a\bar{c}x+\bar{d}u=0
\end{array}\right..$$
In order to have a non trivial solution of the first two equations of the system above, the determinant of the system is zero and so
\begin{equation}
(|a|^2-1)(|d|^2-1)=|c|^2.
\end{equation}
The first equation yields  
\begin{equation}
u=(1-|a|^2)x.
\end{equation}
From the the third equation of the last system, we conclude that $ad\neq0$ and 
\begin{equation}
x=-\dfrac{\bar{d}u}{a\bar{c}}
\end{equation} 
The equation (4.2) and the equation (4.3) imply that
\begin{equation}
 a\bar{c}=(|a|^2-1)\bar{d}.
\end{equation}
By using the fact $c=a^2-ad+d^2$ and the equations (4.1) and (4.4), we get
$$
\bar{d}=-a(\bar{a}-\bar{d})^2\text{ and } \bar{d}(a-d)^2=-a.
$$
Thus we get $$|a-d|^2=1\text{ and }\frac{a}{\bar{d}}=-(a-d)^2.$$
Hence, there exists $\theta\in\mathbb{R}$ such that
 $$a-d=e^{i\theta}\text{ and } a=-\bar{d}e^{2i\theta}.$$
 Therefore we get $$a=d+e^{i\theta}\text{ and }\bar{d}e^{2i\theta}+e^{i\theta}+d=0.$$
 Set $d=\delta'+i\delta$ where $\delta',\delta\in\mathbb{R}$. Then 
\begin{equation}
(1+e^{2i\theta})\delta'+i\delta(1-e^{2i\theta})+e^{i\theta}=0.
\end{equation}
To solve the equation (4.5), consider two cases: (1) $1+e^{2i\theta}=0$,  (2) $1+e^{2i\theta}\neq0$.\\
If $1+e^{2i\theta}=0$ then $d=\delta'\mp\dfrac{i}{2}$ and $a=\delta'\pm\dfrac{i}{2}$.\\
If $1+e^{2i\theta}\neq0$ then $d=-\dfrac{1}{2}\sec\theta-(\tan\theta)\delta+\delta i$.\\
Therefore we have
$$a=d+e^{i\theta}\text{ and } 
\left\{\begin{array}{ll}
d=\delta'-\frac{1}{2}i&\text{ if }\theta=\frac{\pi}{2}\text{ mod }2\pi \\
d=\delta'+\frac{1}{2}i&\text{ if }\theta=-\frac{\pi}{2}\text{ mod }2\pi \\
d=-\frac{1}{2}\sec\theta-(\tan\theta) \delta+i\delta,& \text{ if }\theta\neq\pm\frac{\pi}{2} \text{ mod }2\pi
\end{array}\right.,
$$
where $\theta,\delta',\delta\in\mathbb{R}$.\\
The condition $|a|<1$ follows directly from the equation (4.2) and the fact that $M$ is positive definite. The matrix $M$ becomes $\begin{pmatrix}
1&0\\
0&1-|a|^2
\end{pmatrix}$ up to a constant.
\end{proof}

\begin{prop}
Let $\rho:{\mathcal{NB}}_4\to GL(2,\mathbb{C})$ be the irreducible representation of ${\mathcal{NB}}_4$ given by Proposition 2.1 ($T=\pm i$).
The representation $\rho$ is unitary relative to a hermitian positive definite matrix $M$ if and only if 
$$|d|^2=\dfrac{1}{2}.$$
\end{prop}

\begin{proof}
Recall that 
$$\rho(\tau)=\left(\begin{matrix}
\pm it&0\\
0&t
\end{matrix}\right)~\text{and }~
\rho(\sigma_j)=\left(\begin{matrix}
d&(\pm i)^{j-1}\\
-d^2(\pm i)^{1-j}&d
\end{matrix}\right),$$
where $t$ is an eighth root of unity and $d\neq0$.\\
Suppose that $\rho$ is unitary relative to hermitian positive definite matrix $M$. Then, by Lemma 4.1, we have $M=\left(\begin{matrix}
 x&0\\
 0&u
 \end{matrix}\right)$ and
 $$
\left\{\begin{array}{l}
(|a|^2-1)x+u=0\\
|c|^2x+(|d|^2-1)u=0\\
a\bar{c}x+\bar{d}u=0
\end{array}\right..$$
Since $a=d$ and $c=-d^2$, it follows that
$$\left\{\begin{array}{l}
(|d|^2-1)x+u=0\\
|d|^4x+(|d|^2-1)u=0\\
-d\bar{d}^2x+\bar{d}u=0
\end{array}\right..$$
In order to have a non trivial solution of the first two equations of the last system, the determinant of the system is zero and so $$(|d|^2-1)(|d|^2-1)=|d|^4.$$
The first and third equations of the last system imply that $$x=\frac{u}{d\bar{d}}=\frac{u}{1-|d|^2}$$
Then we get 
$$|d|^2=\frac{1}{2}$$
and the matrix $M$ becomes $\begin{pmatrix}
2&0\\
0&1
\end{pmatrix}$ up to a constant.
\end{proof}
\begin{prop}
Let $\rho:{\mathcal{NB}}_3\to GL(2,\mathbb{C})$ be the irreducible representation of ${\mathcal{NB}}_4$ that is defined, in  proposition 2.1 ($T=e^{\pm2\pi i/3}$).
The representation $\rho$ is unitary relative to a hermitian positive definite matrix $M$ if and only if $|d|<1$ and
$$c=\dfrac{\omega d(|d|^2-1)}{\bar{d}}.$$
\end{prop}

\begin{proof}
Recall that
$$\rho(\tau)=\left(\begin{matrix}
e^{\pm2\pi i/3}t&0\\
0&t
\end{matrix}\right)~\text{and }~
\rho(\sigma_j)=\left(\begin{matrix}
\omega d&e^{\pm(j-1)2\pi i/3}\\
e^{\pm(1-j)2\pi i/3}c&d
\end{matrix}\right),$$
where $t$ is a sixth root of unity and $\omega=e^{\pm i\pi/3}$ such that $c\neq0$.\\
Suppose that $\rho$ is unitary relative to a matrix $M$. Then, by Lemma 4.1, we have $M=\left(\begin{matrix}
 x&0\\
 0&u
 \end{matrix}\right)$  and
 $$
\left\{\begin{array}{l}
(|a|^2-1)x+u=0\\
|c|^2x+(|d|^2-1)u=0\\
a\bar{c}x+\bar{d}u=0
\end{array}\right..$$
Since $a=\omega d$, it follows that
$$\left\{\begin{array}{l}
(|d|^2-1)x+u=0\\
|c|^2x+(|d|^2-1)u=0\\
\omega d\bar{c}x+\bar{d}u=0
\end{array}\right..$$
So
\begin{equation}
x=-\frac{\bar{d}u}{\omega d\bar{c}}=\frac{u}{1-|d|^2}
\end{equation}
Since $\omega^{-1}=\bar{\omega}$, it follows that 
$$c=\frac{\omega d(|d|^2-1)}{\bar{d}}.$$
The condition $|d|<1$ follows from the equation (4.6) and the fact that $M$ is a positive definite.
The matrix $M$ becomes
$\left(
\begin{matrix}
1&0\\
0&1-|d|^2
\end{matrix}\right)$
up to a constant.
\end{proof}


\begin{thebibliography}{99}
\bibitem{Belli} P. Bellingeri, A. Bodin, {\em The braid group of a necklace,} Math. Z. 283 (2016), no. 3-4, 995--1010.
\bibitem{Bir} J. S. Birman, {\em Braids, Links and Mapping Class Groups}, Annals of Mathematical Studies. Princeton University Press. 82 (1975).
\bibitem{Brui} P. Bruillard, L. Chang, S.-M. Hong, J. Plavnik, E. Rowell, M. Sun. {\em Low-dimensional representations of the three component loop braid group}, Journal of Math Phys. 56 (2015), no. 11. 
\bibitem{Bulli}	A. Bullivant, A. Kimball, P. Martin, E. Rowell, {\em Representations of the necklace braid group: topological and combinatorial approaches}, Comm. Math. Phys. 375 (2020), no. 2, 1223--1247.
\bibitem{Kad}  Z. Kadar, P. Martin, E. Rowell, Z. Wang. {\em Local Representations of the Loop Braid Group}, Glasgow Math J. 59 (2017), no. 2. 359-378. 
\end{thebibliography}
\end{document}